\documentclass[11pt]{article}

\usepackage[lmargin=1.0in,rmargin=1.0in,tmargin=1.0in,bmargin=1.0in]{geometry}
\usepackage[utf8x]{inputenc}
\usepackage{amssymb, mathrsfs, amsthm, amsmath}
\usepackage[all]{xy}
\usepackage{url}
\usepackage{bbm}
\usepackage{lmodern}
\usepackage{titlesec}
\usepackage{fancyhdr}
\usepackage[multiple]{footmisc}
\usepackage{paralist}

        {\begin{enumerate}[#1]}{\end{enumerate}%
         \vspace{-.6\baselineskip}}

        {\begin{compactenum}[#1]}{\end{compactenum}}

        {\begin{compactenum}[#1]}{\end{compactenum}\vspace{.5\baselineskip}}

\newcommand{\cA}{\mathcal{A}}

\renewcommand{\AA}{\mathbb{A}}

\newcommand{\CC}{\mathbb{C}}

\newcommand{\FF}{\mathbb{F}}
\newcommand{\GG}{\mathbb{G}}

\newcommand{\cJ}{\mathcal{J}}
\newcommand{\cK}{\mathcal{K}}
\newcommand{\cL}{\mathcal{L}}
\newcommand{\LL}{\mathbf{L}}
\newcommand{\cM}{\mathcal{M}}

\newcommand{\p}{\mathfrak{p}}
\newcommand{\cP}{\mathcal{P}}
\newcommand{\PP}{\mathbb{P}}
\newcommand{\RR}{\mathbb{R}}

\newcommand{\ZZ}{\mathbb{Z}}
\newcommand{\QQ}{\mathbb{Q}}
\newcommand{\OO}{\mathcal{O}}

\newcommand{\cZ}{\mathcal{Z}}

\newcommand{\perf}{^\textsf{perf}}

\newcommand{\Ab}{\textsf{Ab}}

\newcommand{\Cpx}{\textsf{Cpx}}

\newcommand{\Sm}{\textsf{Sm}}

\newcommand{\SmCor}{\textsf{SmCor}}
\newcommand{\SmProj}{\textsf{SmProj}}

\newcommand{\Gal}{\textsf{Gal}}

\newcommand{\ProSS}{\textsf{Pro-}\mathcal{S}}

\newcommand{\HI}{\textsf{HI}}
\newcommand{\NMV}{\textsf{NMV}}

\newcommand{\un}{\mathbbm{1}}

\newtheorem{theo}{Theorem}

\newtheorem{prop}[theo]{Proposition}
\newtheorem*{prop*}{Proposition}
\newtheorem{lemm}[theo]{Lemma}
\newtheorem{coro}[theo]{Corollary}
\newtheorem{conj}[theo]{Conjecture}
\theoremstyle{definition}
\newtheorem{defi}[theo]{Definition}
\newtheorem{rema}[theo]{Remark}

\newtheorem{exam}[theo]{Example}
\newtheorem{warn}[theo]{Warning}
\newtheorem{gues}[theo]{Guess}
\newtheorem{heurDefi}[theo]{{Heuristic ``Definition''}}
\newtheorem{cons}[theo]{Construction}
\newtheorem{obse}[theo]{Observation}

\newtheorem*{theoUn}{Theorem}

\DeclareMathOperator{\id}{id}

\DeclareMathOperator{\Spc}{Spc}
\DeclareMathOperator{\Spec}{Spec}
\DeclareMathOperator{\Proj}{Proj}
\DeclareMathOperator{\SH}{SH}

\DeclareMathOperator{\Ho}{Ho}

\DeclareMathOperator{\supp}{supp}
\DeclareMathOperator{\coker}{coker}
\DeclareMathOperator{\im}{im}
\DeclareMathOperator{\Cone}{Cone}
\DeclareMathOperator{\num}{num}
\DeclareMathOperator{\rat}{rat}

\renewcommand{\top}{\mathsf{top}}

\title{Some observations about motivic tensor triangulated geometry over a finite field}
\author{Shane Kelly}

\begin{document}

\maketitle

\section*{Introduction}

These are notes based on three lectures I gave at the workshop ``Bousfield classes form a set: 
a workshop in memory of Tetsusuke Ohkawa'' at Nagoya University in August 2015. 

The goal of the lectures was to give a brief sketch of the Morel-Voevodsky stable homotopy category $SH(S)$ and motivic stable homotopy groups of spheres, aimed at someone with no previous experience with motives, and then in the last lecture see if anything could be said about the tensor triangulated spectrum $\Spc(SH(S)^c)$ of $SH(S)$. I expected, perhaps naïvely, $\Spc(SH(S)^c)$ to be completely intractable, but to my surprise, it is possible to give a conjectural description of $\Spc(SH(\FF_q)^c_\QQ)$. 

\begin{theoUn}[{Theorem~\ref{theo:main}}]
Let $\FF_q$ be a field with a prime power, $q$, number of elements. Suppose that for all connected smooth projective varieties $X$ we have:
\[
\begin{array}{rllr}
CH^i(X, j)_\QQ &= 0,  & \forall j \neq 0, i \in \ZZ & \textrm{(Beilinson-Parshin conjecture)}, \\
CH^i(X)_\QQ \otimes CH_i(X)_\QQ &\to CH_0(X)_\QQ & \textrm{ is non-degenerate.}
 & \textrm{(Rat.\ and num.\ equiv.\ agree)} 
\end{array}
\]
Then 
\[ \Spc(SH(\FF_q)^c_\QQ)  \cong \Spec(\QQ).\]
\end{theoUn}

Later, I found out there are a considerable number of results about $\Spc(SH(k)^c)$ for $k \subseteq \CC$ in \cite{Joa15}, using methods adapted from the study of the classical stable homotopy category. Another paper studying this object is \cite{HO16}, where it is proven that Balmer's comparison map is surjective (for any field of non-even characteristic!). They also address what is possibly one of the most important questions in this area of study---the production of field spectra.

For a speculative discussion about the structure of $\Spc(SH(\FF_q)^c)$ see Section~\ref{sec:observations}.

In this last section we also prove that Balmer's comparison map
\[ \Spc(SH(\FF_q)^c) \to \Spec^h(K^{MW}(\FF_q)) \]
is surjective in the special case of finite fields of prime characteristic, Corollary~\ref{coro:compMapEpi}. To do this we exploit the fact that we can describe $\Spec^h(K^{MW}(\FF_q))$ completely---an approach due to Ormsby.


\emph{Outline.} The first section contains the basic definitions of tensor triangulated geometry, and references to the literature. It is aimed at the motivic reader who has had minimal to no exposure to these things.

The second section contains basic definitions of various motivic categories, and references to the literature. It is aimed at the tensor triangulated geometer who has had minimal to no exposure to these things.

The last section contains some observations and guesses about the structure of $\Spc(SH(\FF_q)^c)$. In particular, it contains Theorem~\ref{theo:main} and its proof.

\emph{Acknowledgements.} I thank the organisers of the conference ``Bousfield classes form a set: 
a workshop in memory of Tetsusuke Ohkawa'' for the invitation to speak which led me to think about these things, and also for having organised such an interesting conference. I also thank Paul Balmer, Jens Hornbostel, and Denis-Charles Cisinski for interesting discussions about potential future work, Marc Hoyois for discussions about the étale homotopy type, and Jeremiah Heller and Kyle Ormsby for pointing out that an ``elementary fact'' I was using in the proof of Proposition~\ref{prop:epi} is actually a combination of theorems of Ayoub, Balmer, Gabber, and Riou.

\section{Tensor triangulated geometry}

In this section we recall some basic definitions from tensor triangulated geometry. For a much more readable exposition of this material the reader is encouraged to consult \cite{BalTTG10}. 

\begin{exam}
The most enlightening example to keep in mind when reading the following definitions is the bounded derived category of coherent sheaves on smooth variety $X$ over a field, or equivalently, the derived category $D\perf(X)$ of perfect complexes on $X$. Of course, there are other important examples coming from scheme theory, stable homotopy theory, modular representation theory, noncommutative topology, and, as we shall see below, the theory of motives, \cite[\S1]{BalTTG10}.
\end{exam}

\begin{warn}
Due to the similarity in the definition of the spectrum $\Spec(R)$ 
of a commutative ring $R$, 
and the spectrum $\Spc(\cK)$ of a $\otimes$-triangulated category $\cK$, 
it is tempting to try and blindly transfer all notions from commutative algebra to tensor triangulated geometry. However, this similarity is a red herring. Whereas the elements of a ring $R$ 
behave like functions, the objects in a $\otimes$-triangulated $\cK$ 
behave like (bounded complexes of) sheaves: a closed subset of $\Spec(R)$ corresponds to the set of functions vanishing on it, whereas a closed subset of $D\perf(\Spec(R))$ corresponds to the set of (perfect complexes of) sheaves supported on it. The first manifestation of this is that the correspondence between ideals and closed subsets is inclusion preserving in the $\otimes$-triangulated world, where it is inclusion reversing in the world of commutative algebra. Cf. also Proposition~\ref{prop:firstProp}\eqref{prop:firstProp:open} and \eqref{prop:firstProp:closed} below. For a more surprising and serious example of this phenomenon see \cite[Rmk{.}27]{BalTTG10}.
\end{warn}

\begin{defi} \label{defi:basicDef}
Let $(\cK, \otimes, \un)$ be an essentially small $\otimes$-triangulated category, i.e., a triangulated category equipped with a monoidal structure $\cK \times \cK \stackrel{\otimes}{\to} \cK$ with a unit object $\un$, such that $\otimes$ is exact in each variable, \cite[Def{.}3]{BalTTG10}. Let $\cJ$ be a non-empty full triangulated subcategory.
\begin{enumerate}
 \item $\cJ$ is called \emph{thick} if it is stable under direct summands; $a \oplus b \in \cJ \Rightarrow a \in \cJ$ or $b \in \cJ$, \cite[Def{.}7]{BalTTG10}. For any closed subvariety $Z$ of a smooth variety $X$, the subcategory of objects supported on $Z$, i.e., objects sent to zero by the canonical functor $D\perf(X) \to D\perf(X{-}Z)$, is thick.

 \item $\cJ$ is said to be a \emph{$\otimes$-ideal} if $\cK \otimes \cJ \subseteq \cJ$, \cite[Def{.}7]{BalTTG10}. The subcategories $ker(D\perf(X) \to D\perf(X{-}Z))$ just mentioned are tensor ideals. For an example of a thick subcategory which is not a tensor ideal, the reader could consider using the Fourier-Mukai transform \cite[Prop{.}9.19]{Huy06} between the derived category of an abelian variety and its dual, as this preserves thick subcategories, but not the tensor structure. 

 \item A \emph{prime} of $\cK$ is a thick $\otimes$-ideal $\cP$ such that $\un \notin \cP$ and $a \otimes b \in \cP \Rightarrow a \in \cP$ or $b \in \cP$, \cite[Constr{.}8]{BalTTG10}. The set of primes is denoted by $\Spc(\cK)$. The ideals $ker(D\perf(X) \to D\perf(X{-}Z))$ mentioned above are prime if and only if $Z$ is irreducible.

 \item \label{defi:basicDef:supp} The \emph{support}, denoted by $\supp(a)$, of an object $a \in \cK$ is the set of primes not containing it,
 \[ \supp(a) = \{ \cP \in \Spc(\cK) : a \notin \cP \}. \]
 The complement of $\supp(a)$ is denoted by $U(a) = \{ \cP \in \Spc(\cK) : a \in \cP \}$, \cite[Constr{.}8]{BalTTG10}.

 \item The set $\Spc(\cK)$ has a canonical topology with basis the sets $U(a)$ as $a$ ranges over all objects in $\cK$, \cite[Constr{.}8]{BalTTG10}. 
 
 \item To a subset $Y$ of $\Spc(\cK)$, we associate the full subcategory 
 \[ \cK_Y = \{ a \in \cK : \supp(a) \subseteq Y \}. \]
 If $Y$ is a union $Y = \cup_{i \in I} Y_i$ of subsets $Y_i$ whose complement $\Spc(\cK) - Y_i$ is open and \emph{quasi-compact} (in the sense that every open cover admits a finite subcover), then $\cK_Y$ is a thick $\otimes$-ideal of $\cK$, \cite[Rmk{.}12, Thm{.}14]{BalTTG10}.
 
 \item If $U$ is a quasi-compact open of $\Spc(\cK)$, with closed complement $Z = \Spc(\cK) - U$, one defines the tensor triangulated category
 \[ \cK(U) = (\cK / \cK_Z)^\natural \]
 as the idempotent completion of the Verdier localisation of $\cK$ by $\cK_Z$. An inclusion $V \subseteq U$ of such opens comes with a canonical $\otimes$-triangulated functor $\cK(U) \to \cK(V)$, \cite[Constr{.}24]{BalTTG10}.

 \item The sheafification of the assignment sending a quasi-compact open $U$ to $\hom_{\cK(U)}(\un, \un)$ is denoted by $\OO_\cK$. This gives $\Spc(\cK)$ the structure of a \emph{ringed space}, and, at least when $\cK$ is rigid%
\footnote{\label{footnote:rigid}An object $a$ in a $\otimes$-category is called \emph{strongly dualisable} if there exists an object $Da$ such that $a \otimes-$ is left adjoint to $(Da) \otimes-$. A $\otimes$-category is called \emph{rigid} if every object is strongly dualisable.} %
 and idempotent complete, $(\Spc(\cK), \OO_K)$ is a \emph{locally ringed space}, \cite[Cor{.}6.6]{BalSSS10}. This ringed space is referred to as the \emph{spectrum} of $\cK$, \cite[Constr{.}29]{BalTTG10}.
\end{enumerate}
\end{defi}

\begin{exam}\ \label{exam:spc}
\begin{enumerate}
 \item Let $X$ be a quasi-compact quasi-separated scheme (for example a variety over a field). Then there is an isomorphism of locally ringed spaces $X \cong \Spc(D^{\textsf{perf}}(X))$ where $D^{\textsf{perf}}(X)$ is the derived category of perfect complexes on $X$, \cite[Thm.54]{BalTTG10}.

 \item Let $SH_\top$ be the classical stable homotopy category. Then $\Spc(SH_\top^c)$ is
 \begin{equation} \label{equa:spcSHtop}
  \xymatrix@!=6pt{
\cP_{2, \infty} \ar@{-}[d] & \cP_{3, \infty} \ar@{-}[d] & \cP_{5, \infty} \ar@{-}[d] & \cP_{7, \infty} \ar@{-}[d] & \dots  \\
\vdots \ar@{-}[d] & \vdots \ar@{-}[d] &\vdots \ar@{-}[d] &\vdots \ar@{-}[d] &   \\
\cP_{2, 2} \ar@{-}[d] & \cP_{3, 2} \ar@{-}[d] & \cP_{5, 2} \ar@{-}[d] & \cP_{7, 2} \ar@{-}[d] & \dots  \\
\cP_{2, 1} \ar@{-}[drr] & \cP_{3, 1} \ar@{-}[dr] & \cP_{5, 1} \ar@{-}[d] & \cP_{7, 1} \ar@{-}[dl] & \dots  \ar@{-}[dll] & \ar@{-}[dlll]  \\
&& \cP_{0,1} &&&
 } 
 \end{equation}
The lines indicate that the higher prime is in the closure of the lower one. For every prime number $p$ and every $n \geq 1$, the prime $\cP_{p,n}$ of $SH_\top^c$ s the kernel of the $n$th Morava $K$-theory (composed with localisation at $p$) and $\cP_{p, \infty} = \cap_{n \geq 1} \cP_{p,n}$ is the kernel of localisation at $p$. The generic point $\cP_{0,1} = (SH_\top^c)_\textsf{tor} = Ker(H(-, \QQ))$ is the kernel of singular cohomology with $\QQ$-coefficients, Hopkins-Smith \cite{HS98}, \cite[Cor.9.5]{BalSSS10}, \cite[Thm.51]{BalTTG10}.

 \item \label{exam:spc:equi} Let $G$ be a finite group, and let $SH_G$ be the $G$-equivariant stable homotopy category. For a subgroup $H \subseteq G$ let $\Phi^H: SH_G^c \to SH_\top^c$ denote the geometric $H$-fixed points functor. Then, \emph{as a set}, 
 \[ 
 \Spc(SH_G^c) = \left \{ \cP(H, p, n) \stackrel{def}{=} (\Phi^H)^{-1}\cP_{p,n} : H \leq G, \cP_{p,n} \in \Spc(SH_\top^c) \right \}. 
 \]
  Furthermore, $\cP(H, p, n) = \cP(H', p', n')$  if and only if $H$ and $H'$ are conjugate in $G$, and $p = p', n = n'$. For more details see \cite{BS15}.
  
  In the case of a cyclic group $G = \ZZ / n$, the space $\Spc(\SH_G^c)$ contains a copy of $\Spc(\SH_\top^c)$ for every $m$ dividing $n$, including $1$ and $n$. Over $\Spec(\ZZ[1 / n])$, the copies are disjoint, but there are some specialisation-generisation relations between the points lying over $\Spec(\ZZ / n)$, cf. \cite[Eq.1.3]{BS15}.
\end{enumerate}
\end{exam}

Just as schemes admit a canonical comparison morphism to the spectrum of their global sections, there are canonical comparison morphisms from $\Spc(\cK)$ to the spectrum of the ring of endomorphisms of the unit object.

\begin{theo}[{\cite[Thm{.}5.3, Cor{.}5.6, Thm{.}7.13, Not{.}3.1]{BalSSS10}}] \label{theo:structureMap}
Let $\cK$ be an essentially small $\otimes$-triangulated category and $u \in \cK$ an invertible object. There are two continuous maps of topological spaces
\[ \rho_\cK: \Spc(\cK) \to \Spec \biggl ( \hom_\cK(\un, \un) \biggr ), \qquad \rho^\bullet_\cK: \Spc(\cK) \to \Spec^h \biggl ( \oplus_{n \in \ZZ} \hom_\cK(\un, u^n) \biggr ). \]
Here, $\Spec^h$ indicates the set of proper homogeneous ideals which are prime.\footnote{Recall that $\Proj$ of a non-negatively graded ring is the set of those proper homogeneous prime ideals \emph{which don't contain all elements of positive degree}. There is no such exclusion in the definition of $\Spec^h$.} It is equipped with the topology whose closed sets are of the form $V(I^\bullet) = \{ \p^\bullet \in \Spec^h : I^\bullet \subseteq \p^\bullet \}$ for homogeneous ideals $I^\bullet$, \cite[Rmk{.}3.4]{BalSSS10}.

Futhermore, if $\hom_{\cK}(\un[i], \un) = 0$ for $i < 0$, then $\rho_\cK$ is surjective. In the case $u = \un[1]$, and the graded endomorphism ring is coherent (e.g., noetherian), the map $\rho^\bullet_\cK$ is surjective.

For any morphism $s: \un \to u^n$, the premiage of the principle open $D(s)$ of homogeneous primes not containing $s$ is the open $U(Cone(s))$ of primes of $\cK$ containing $Cone(s)$.
\end{theo}

\begin{rema} \label{rema:structureMapDef}
The maps of Theorem~\ref{theo:structureMap} are as follows. The first one is defined on primes $\cP$ by $\rho_\cK(\cP) = \{ \un \stackrel{f}{\to} \un : Cone(f) \not\in\cP \}$. The second one takes a prime $\cP$ to the homogeneous ideal generated by those $\un \stackrel{f}{\to} u^n$ such that $Cone(f) \not\in\cP$ as $n$ ranges over all integers.
\end{rema}

Later on we will use the following facts.

\begin{prop}[Balmer] \label{prop:firstProp}
Let $\cK$ be an essentially small $\otimes$-triangulated category.
\begin{enumerate}
 \item If $F: \cK \to \cL$ is a $\otimes$-exact functor to another essentially small $\otimes$-triangulated category $\cL$, then the assignment $\cP \mapsto F^{-1}\cP$ defines a continuous map of topological spaces $\Spc(F): \Spc(\cL) \to \Spc(\cK)$. For every $a \in \cK$, we have $\supp(F(a)) = \Spc(F)^{-1}(\supp(a))$, \cite[Prop.11(c)]{BalTTG10}.

 \item \label{prop:firstProp:open} Let $\cJ \subset \cK$ be a thick $\otimes$-ideal. Then the Verdier localisation $\cK \to \cK / \cJ$ induces a homeomorphism from $\Spc(\cK / \cJ)$ onto the subspace $\{ \cP : \cJ \subseteq \cP \}$ of $\Spc(\cK)$. This subspace is not always \emph{open} (although it is if, for example, $\cJ$ is generated by a single element), but it is always \emph{closed under specialisation}, \cite[Thm.18(a)]{BalTTG10}.

 \item \label{prop:firstProp:closed} Let $u \in \cK$ be an object such that the cyclic permutation $u^{\otimes 3} \stackrel{\sim}{\to} u^{\otimes 3}$ is the identity. Then the localisation $\cK \to \cK[u^{\otimes -1}]$ induces a homeomorphism from $\Spc(\cK[u^{\otimes -1}])$ onto the \emph{closed} subspace $\supp(u) = \{ \cP : u \notin \cP \}$ of $\Spc(\cK)$, \cite[Thm.18(c)]{BalTTG10}.

 \item \label{prop:firstProp:nil} Let $f$ be a morphism between tensor invertible objects of $\cK$. Then there exists a prime containing $Cone(f)$ if and only if $f^{\otimes n} \neq 0$ for all $n \geq 0$, 
\cite[Thm.2.15]{BalSSS10}, 
 \cite[Cor.2.5]{Bal05}.
\end{enumerate}
\end{prop}

The following notion of prime ideal of an abelian category which is analogous to primes of a $\otimes$-triangulated category was developed in \cite{Pet13}.

\begin{defi}[{cf. \cite[Def{.}4.2]{Pet13}}]
Let $\cA$ be an abelian category. Recall that a full\footnote{This assumption does not appear in \cite[Def{.}4.2]{Pet13}, but it seems it should be there.} subcategory of $\cA$ is called \emph{thick}, \cite[Def{.}8.3.21.(iv)]{KS05}, if it is closed under extensions, kernels, and cokernels.%
\footnote{In \cite[Def{.}4.2]{Pet13}, the definition uses \emph{coherent} abelian subcategories, which, as Oliver Braunling pointed out to me, are just thick subcategories containing a zero object.} %
Suppose now that $\cA$ is a tensor abelian category. A \emph{(thick) tensor ideal} of $\cA$ is a (full) thick subcategory $\cM \subseteq \cA$ such that $\cM \otimes \cA \subseteq \cM$. A proper $\otimes$-ideal $\cM \subset \cA$ is called a \emph{(thick) prime ideal} if $a \otimes b \in \cM$ implies $A \in \cM$ or $B \in \cM$.
\end{defi}

\begin{lemm} \label{lemm:rigssfield}
Suppose that $(\cA, \otimes, \un)$ is a rigid$^\textrm{\ref{footnote:rigid}}$ %
semisimple%
\footnote{An object $a$ in an abelian category is called \emph{simple} if every monomorphism $b \to a$ (in the categorical sense) is either zero or an isomorphism. An object is \emph{semisimple} if it is a sum of simple objects. An abelian category is \emph{semisimple} if all of its objects are semisimple. } %
tensor abelian category. %
If $\un$ is simple, then both $K^b(\cA)$ and $\cA$ possess a unique prime: $\{0\}$.
\end{lemm}

\begin{proof}
Since $\cA$ is semisimple, every object of $K^b(\cA)$ is isomorphic to the sum of its shifted cohomology objects, i.e., $K^b(\cA) \cong \bigoplus_{i \in \ZZ} \cA$ as a triangulated category; the triangulated structure on $\bigoplus_{i \in \ZZ} \cA$ is given by shifting the indices, and $\Cone(a \stackrel{f}{\to} b) = \ker(f)[1] \oplus \coker(f)$. Inspecting the definitions, we see that $\cM \mapsto \bigoplus_{i \in I} \cM$ is a bijection from the set of prime ideals of the tensor abelian category $\cA$ to the set of primes of the $\otimes$-triangulated category $\bigoplus_{i \in \ZZ} \cA$. So it suffices to treat the case of $\cA$.

Let $a$ be a nonzero object of some prime ideal $\cM$. Since it is a tensor ideal, it must also contain $(Da) \otimes a$ where $Da$ is the strong dual of $a$, which exists by the assumption that $\cA$ is rigid. By definition, $(Da) \otimes-$ is right adjoint to $a \otimes- $, and so we have canonical morphisms $\un \stackrel{\epsilon}{\to} (Da) \otimes a$ and $a \otimes (Da) \stackrel{\eta}{\to} \un$ such that the composition $a \stackrel{\id_a \otimes \epsilon}{\longrightarrow} a \otimes (Da) \otimes a \stackrel{\eta \otimes \id_a}{\longrightarrow} a$ is $\id_a$. If $\epsilon$ were to be zero, then $\id_a$ would be zero, contradicting the assumption that $a$ is nonzero. So $\epsilon$ is nonzero, and by the assumption that $\un$ is simple, and $\cA$ is semisimple, $\epsilon$ must be the inclusion of a direct summand. Since prime ideals are closed under direct summand, it follows that $\cM$ contains $\un$, and therefore $\cM = \cA$.
\end{proof}

\section{Motivic categories}

In this section we rapidly review the motivic categories that we will discuss. Specifically, the Morel-Voevodsky stable homotopy category $SH(S)$, Voevodsky's triangulated category of motives $DM(S)$, Grothendieck's classical categories of motives with respect to an adequate equivalence relation $\cM_\sim(k)$, and some of the relationships between these. This section will be too basic for the experts, and too terse for the non-experts, but we hope that it will at least serve to set notation for the experts, and provide references to the literature for the non-experts.

\begin{rema}
Philosophically, categories of motives should be defined by universal properties. Consequently, all the constructions have a ``generators and relations'' feel to them, cf. Remark~\ref{rema:ClassicalRelations}, Remark~\ref{rema:voevDMgenrel}, and Definition~\ref{defi:SHinfinity}.
\end{rema}

\subsection{Grothendieck motives}

A nice introduction to classical Grothendieck motives over a field is \cite{Sch94}. For the extension to a smooth base and a beautiful application of this extension see \cite{DM91}.

\begin{defi}
\emph{Cycle groups}. Let $k$ be a field, and $S$ a smooth $k$-scheme of pure dimension $d_S$. For a smooth projective $S$-scheme $X$, let $\cZ^i(X)$ denote the free abelian group generated by the closed integral subscheme of $X$ of codimension $i$. If $\sim$ is an adequate equivalence relation%
\footnote{\label{footnote:adEqRel}An \emph{adequate equivalence relation} is a family of equivalence relations $\sim_X$ on the $\cZ^\ast(X)$ which satisfy three properties, which essentially require that composition as defined above is well-defined, \cite{Sam60}. In short, pullback, pushforward, and intersection are well-defined.} %
such as rational equivalence,%
\footnote{Two cycles $\alpha, \alpha' \in \cZ^i(X)$ are rationally equivalent if there is a cycle $\beta \in \cZ^i(\PP^1_X)$ such that $\beta \cdot [\{0\} \times X] = \alpha$ and $\beta \cdot [\{\infty\} \times X] = \alpha'$. Rational equivalence is the coarsest equivalence relation.} %
homological equivalence,%
\footnote{A cycle $\alpha$ is homologically equivalent to zero if its image under the cycle class map $\cZ^i(X) \to H^{2i}(X)$ is zero, for some prechosen Weil cohomology theory, such as étale cohomology $H^i(X) = H_{et}^{2i}(X, \QQ_l(i))$ for some $l \nmid \mathsf{char}\ p$. In other words, $A_{\hom}^i(X)$ is the image of the cycle class map $A_{\hom}^i(X) = \im(\cZ^i(X) \to H^{2i}(X))$.} %
or numerical equivalence,%
\footnote{Numerical equivalence is the coarsest equivalence relation which makes the intersection product $A_{\num}^i(X) \otimes A_{\num}^{d - i}(X) \to A_{\num}^d(X)$ nondegenerate, where $d = \dim X$. That is, $\alpha \in A^i_{\rat}(X)$ is numerically equivalent to zero if and only if $\alpha \cdot \beta = 0$ for all $\beta \in A^{d - i}_{\rat}(X)$.} %
denoted $\rat, \hom$, and $\num$ respectively, we will write $A_\sim^i(X) = \cZ^i(X) / \sim$.

\emph{Cycle categories.} For any triple $X, Y, Z$ of smooth projective $S$-schemes, and cycles $\alpha \in A_\sim^i(X\times_S Y)$, and $\beta \in A_\sim^j(Y\times_SZ)$, the composition $\beta \circ \alpha \in A_\sim^{i + j}(X\times_S Z)$ is defined by pulling $\alpha$ and $\beta$ back to the triple product $X \times_S Y \times_S Z$ along the canonical projections, intersecting them, and then obtaining a cycle on $X \times_S Z$ by pushing forward along the canonical projections. 
In this way we obtain a category, whose objects are smooth projective $S$-schemes, and $\hom(X, Y) = A_\sim^{\dim X/S}(X\times_SY)$. Here $\dim X / S$ denotes the relative dimension of the morphism $X \to S$. Identity morphisms are given by the cycles associated to the diagonals $\Delta_X \subseteq X \times_S X$. Fibre product of $S$-schemes induces a tensor product on this category.

\emph{Motivic categories.} The objects of $\cM_\sim(S)$ are triples $(X, p, n)$ where $X$ is a smooth projective $S$-scheme, $p \in A_\sim^{\dim X/S}(X \times_S X)$ satisfies $p \circ p = p$, and $n \in \ZZ$. We set
\[ \hom_{\cM_\sim(S)}((X, p, n), (Y, q, m)) = \{ \alpha \in A_\sim^{\dim X/S + n - m}(X \times_S Y) : \alpha = p \circ \alpha \circ q\}. \]
This category is a tensor additive category with sum induced by disjoint union of $S$-schemes, and tensor product induced by fibre product. There is a canonical functor 
\[ M: \SmProj(S) \to \cM_\sim(S); \qquad X \mapsto (X, [\Delta_X], 0) \] 
from smooth projective $S$-schemes which sends a morphism $f: X \to Y$ to the cycle associated to its graph $\Gamma_f \subseteq X \times_S Y$.
\end{defi}

\begin{rema}
Any section $s: S \to X$ of an $S$-scheme $f:X \to S$, for example the section at infinity $\infty: S \to \PP^1_S$, gives rise to any idempotent endomorphism $s \circ f$, and consequently a decomposition $M(X) \cong M(S) \oplus (\ker M(s \circ f))$. The Leftschetz motive $\LL$ is defined as the kernel of the projection to infinity $M(\PP^1_S) \cong M(S) \oplus \LL$. We then have a canonical isomorphism 
\[ (X, p, n) \cong \biggr (\im \bigr (M(X) {\stackrel{p}{\to}} M(X) \bigl ) \biggl )\otimes \LL^{\otimes (- n)}. \]
\end{rema}

\begin{rema} \label{rema:ClassicalRelations}
This definition can be seen as a generators and relations construction of a category. The generators are smooth projective varieties and correspondences, and the relations we have forced are the equivalence relation $\sim$, the existence of kernels and images of idempotent endomorphisms, and the tensor inverse of $\LL$.
\end{rema}

\begin{defi} \label{defi:tensorLambda}
For an abelian group $A$ and $\Lambda$ a flat $\ZZ$-algebra, we write $A_\Lambda = A \otimes \Lambda$. If $\cA$ is an additive category, we write $\cA_\Lambda$ for the category which has the same objects as $\cA$ and $\hom_{\cA_\Lambda}(a, b) = \hom_\cA(a, b)_\Lambda$.
\end{defi}

\begin{rema}
Since $\QQ$ is a flat $\ZZ$-algebra, the categories $\cM_\sim(k)_\QQ$ are the same as the ones constructed as above, using $A_\sim^\ast(X)_\QQ$ instead of $A_\sim^\ast(X)$.
\end{rema}

\begin{theo}[{\cite[Thm{.}1]{Jan92}}] \label{theo:jannsen} 
Let $k$ be a field and $\sim$ an adequate equivalence relation. The category $\cM_{\sim}(k)_\QQ$ is semisimple if and only if $\sim$ is numerical equivalence.
\end{theo}

\begin{rema}[{\cite[1.15]{Sch94}}] \label{rema:rigidity}
There is a canonical functor ${}^\vee: \cM_\sim(k)^{op} \to \cM_\sim(k)$ defined by $(X, p, m)^\vee = (X, {}^t, d - m)$ when $X$ is purely $d$ dimensional. Here ${}^t: A^{d}_\sim(X \times_S X) \to A^{d}_\sim(X \times_S X)$ corresponds to swapping the $X$s.

One can verify by hand that ${M^\vee}^\vee = M$ and that for any three objects $M, N, P$ one has
\[ \hom(M \otimes N, P) = \hom(M, N^\vee \otimes P). \]
Consequently, the categories $\cM_\sim(k)$ are rigid.
\end{rema}

\subsection{Voevodsky motives}

\begin{defi}
\emph{Correspondences.} Let $S$ be a regular%
\footnote{\label{footnote:smcordisclaimer}The category $\SmCor(S)$ can be defined for any noetherian separated scheme $S$, but we have not mentioned this construction because branch points make composition more subtle. For a more general $S$, the group $\hom_{\SmCor(S)}(X, Y)$ is only a proper subgroup of the free abelian group we describe, since branches of $X$ introduce an ambiguity in the composition of some cycles. In fact, it can be defined as the largest subgroup for which composition is well-defined, as soon as the notion of composition has been formalised appropriately, \cite[Chap{.}2]{Kel12}.} %
 noetherian separated scheme, such as the spectrum of a field. The category $\SmCor(S)$ has as objects smooth $S$-schemes, and $\hom_{\SmCor(S)}(X, Y)$ is the free abelian group generated by closed integral subschemes $Z \subseteq X {{\times_S}} Y$ such that the map $Z \to X$ induced by projection is finite and surjective. Composition of two morphisms $\alpha \in \hom_{\SmCor(S)}(X, Y)$ and $\beta: \hom_{\SmCor(S)}(Y, W)$ is defined as above: by pulling back to the triple product $X {\times_S} Y {\times_S} W$, intersecting, and then pushing forward to $X {{\times_S}} W$. The condition that the generators of the hom groups be finite and surjective over the first component ensures that the two pullbacks to $X {\times_S} Y {\times_S} W$ intersect properly. The category $\SmCor(S)$ is a tensor additive category, with direct sum induced by disjoint union of schemes, and tensor product induced by fibre product. As above, there is a canonical functor
\[ [-]: \Sm(S) \to \SmCor(S) \]
which sends a morphism to the cycle induced by its graph.

\emph{Effective geometric motives.} Since $\SmCor(S)$ is additive, we can consider its bounded homotopy category $K^b(\SmCor(S))$. Define $\HI$ to be the set of complexes of the form
\[ (\dots \to 0 \to [\AA^1_X] \to [X] \to 0 \to \dots) \]
 where $X$ ranges over all smooth $S$-schemes. Define $\NMV$ to be the set of complexes of the form
\[ (\dots \to 0 \to [U{\times_X}V] \stackrel{[k] + [l]}{\to} [U] \oplus [V] \stackrel{[i] - [j]}{\to} [X] \to 0 \to \dots ) \]
ranging over all Nisnevich distinguished squares.%
\footnote{\label{footnote:nisDistSq} A \emph{Nisnevich distinguished square} is a cartesian square
$\underset{U}{\stackrel{U{\times}_XV}{\downarrow}}\underset{\to}{\stackrel{\to}{\vphantom{\downarrow}}}\underset{X}{\stackrel{V}{\downarrow}}$ %
such that $U \to X$ is an open immersion, $V \to X$ is étale, and $(X {-} U)_{red}{\times}_X V \to (X {-} U)_{red}$ is an isomorphism.
}%
\footnote{In \cite{Voe00} Voevodsky only uses Zariski distinguished squares, i.e., those squares for which $j$ is also an open immersion. However, \cite[Thm{.}3.1.12]{Voe00} implies that, at least when the base is a perfect field, the Zariski and Nisnevich versions produce the same category. On the other hand, it is the Nisnevich descent property which is often used in most of the proofs in \cite{Voe00}. Nisnevich locally, closed immersions of smooth schemes look like zero sections of trivial affine bundles, cf. \cite[Proof of Lemma 2.28]{MV99}.}
The category of \emph{effective geometric motives} is defined as
\[ DM_{gm}^{eff}(S) = \biggl ( \frac{K^b(\SmCor(S))}{\langle \HI \cup \NMV \rangle }\biggr )^\natural \]
where $(-)^\natural$ denotes idempotent completion, $\langle \HI \cup \NMV \rangle$ is the smallest thick triangulated subcategory generated by the sets of objects $\HI$ and $\NMV$, and the fraction denotes the Verdier quotient.

\emph{Noneffective geometric motives.} The tensor structure on $\SmCor(S)$ induces a canonical tensor structure on $DM_{gm}^{eff}(S)$, and there is a canonical monoidal functor 
\[ M: \Sm(S) \to DM_{gm}^{eff}(S) \]
induced by $[-]: \Sm(S) \to \SmCor(S)$. As above, projection to infinity defines an idempotent endomorphism of $M(\PP^1_S)$ the Tate motive $\ZZ(1)[2]$ to be the kernel of this endomorphism: $M(\PP^1_S) \cong M(S) \oplus \ZZ(1)[2]$. Since we are working with complexes, we have an explicit model for this:
\[ \ZZ(1) \stackrel{def}{=} (\dots \to 0 \to \underset{2}{[\PP^1_S]} \to \underset{3}{[S]} \to 0 \to \dots) \]
as a complex concentrated in degrees 2 and 3. We obtain the category of \emph{noneffective geometric motives} is defined by forcing $\ZZ(1)$ to be tensor invertible. 
\[ DM_{gm}(S) = DM_{gm}^{eff}(S)[\ZZ(1)^{-1}]. \]
Formally, the objects of $DM_{gm}(S)$ are pairs $(M, m)$ with $M$ an object of $DM_{gm}^{eff}(S)$ and $m \in \ZZ$, and morphisms are $\hom((M, m), (N,n)) = \varinjlim_{k \geq m, n}(M \otimes \ZZ(1)^{k -m}, N \otimes \ZZ(1)^{k -n})$.
\end{defi}

\begin{rema}\ 
\begin{enumerate}
 \item To simplify notation, one usually writes $\ZZ(m) = \ZZ(1)^{\otimes m}$ and $M(m) = M \otimes \ZZ(m)$.

 \item To show that the tensor structure on $DM_{gm}(S)$ is well defined on morphisms, one needs to show that the cyclic permutation $\ZZ(1)^{\otimes 3} \stackrel{\sim}{\to} \ZZ(1)^{\otimes 3}$ is the identity. This is \cite[Cor{.}2.1.5]{Voe00}.
 
 \item The transition morphisms in the colimit defining the hom groups in $DM_{gm}(S)$ are actually all isomorphisms, at least when the base is a perfect field: Voevodsky shows in the ``incredibly short and absolutely ingenious''%
 \footnote{In his MathSciNet review, R{\"o}ndigs attributes this quote to Suslin.} %
 article \cite{Voe10} that $-\otimes \ZZ(1)$ is a fully faithful functor on $DM_{gm}^{eff}(S)$ when $S$ is the spectrum of a perfect field.
\end{enumerate}
\end{rema}

\begin{rema} \label{rema:voevDMgenrel}
Again, this definition is clearly of the generators and relations form. One starts with smooth schemes and correspondences as the generators, and then forces $\AA^1$-invariance, Nisnevich descent, kernels and cokernels of idempotents to exist, and $\ZZ(1)$ to be tensor invertible.
\end{rema}

The most important facts we will need about $DM_{gm}(S)$ are the following.

\begin{theo} \label{theo:dmtheo}
Let $k$ be a perfect field of exponential characteristic $p$.
\begin{enumerate}
 \item The category $DM_{gm}(k)_{\ZZ[1/p]}$ is rigid. If $X$ is a smooth projective $k$ variety of pure dimension $d$, then the dual of $M(X)$ is $M(X)(-d)[-2d]$, \cite[Thm{.}5.5.14]{Kel12}, \cite[Thm{.}4.3.7]{Voe00}.

 \item \label{theo:dmtheo:2} For any smooth $k$-variety $X$, $n \in \ZZ$, $i \geq 0$, there are canonical isomorphisms
 \[ \hom_{DM_{gm}(k)}(M(X), \ZZ(i)[n]) \cong CH^i(X, 2i - n) \]
towards Bloch's higher Chow groups, defined in \cite{Blo86}. In particular, for any smooth $k$-variety $X$ and smooth projective $k$-variety $Y$ of pure dimension $d$, there are canonical isomorphisms
 \[ \hom_{DM_{gm}(k)_{\ZZ[1/p]}}(M(X), M(Y)(i)[n]) \cong CH^{i + d}(X{\times_k}Y, 2i - n)_{\ZZ[1/p]}, \]
\cite[Thm{.}5.6.4]{Kel12}, \cite[Thm{.}19.1]{MVW}, \cite[Thm{.}3.2]{Sus00}.

 \item \label{theo:dmtheo:3} Consequently, there is a canonical fully faithful tensor additive embedding
 \[ \cM_{\rat}(k)_{\ZZ[1/p]} \to DM_{gm}(k)_{\ZZ[1/p]}. \]

 \item \label{theo:dmtheo:4} The category $DM_{gm}(k)_{\ZZ[1/p]}$ is generated (as an idempotent complete tensor triangulated category) by motives of smooth projective $k$-varieties. \cite{Bon11}, \cite[Cor{.}3.5.5]{Voe00}.
\end{enumerate}
\end{theo}

\begin{rema}
If one believes in strong resolution of singularities, in the sense of \cite[Def{.}3.4]{FV00}, then one doesn't have to invert $p$ in the above theorem.
\end{rema}

Using the fact that étale cohomology has a structure of transfers, is homotopy invariant, satisfies Nisnevich descent, and is $\PP^1$-stable, one can construct a canonical functor to $l$-adic sheaves.

\begin{theo}[{\cite[Thm{.} 3.1]{Ivo06}, cf. also \cite[Appendix A]{KS}}] \label{theo:etalereal}
For any noetherian separated scheme $S$, and $l$ invertible on $S$, and any $n > 0$ there are canonical tensor triangulated functors
\[ DM_{gm}(S) \to D_{et}(S, \ZZ / l^n), \qquad DM_{gm}(S)_\QQ \to D_{et}(S, \QQ_l). \]
Here, the target categories are the derived categories associated to the small étale site of $S$. If $f: X \to S$ is a smooth morphism, then the image of $M(X)$ is the pushforward $Rf_*(\ZZ / l^n)_X$, resp. $Rf_*(\QQ_l)_X$, of the constant sheaf on the small étale site of $X$. If $S$ is the spectrum of a perfect field, then this also holds for non-smooth $X$.
\end{theo}


\subsection{Morel-Voevodsky's stable homotopy category}

The definition of $SH(S)$ was sketched in \cite{Voe98} using the unstable theory of \cite{MV99}. A more explicit construction, which includes the use of symmetric spectra is in \cite{Jar00}, and a more modern treatment which incorporates advances in the theory of model categories which happened in the meantime (many of which were motivated precisely for the study of $SH(S)$) appears in \cite[Def{.}4.5.52]{Ayo07}. Even more recently, the universal property which $SH(S)$ satisfies was made formal in \cite[Cor{.}2.39]{Rob15}, using the language of infinity categories. Below we sketch the construction of \cite{Cis13}, which we find to be the most accessible.

\begin{heurDefi} \label{defi:SHinfinity}
Let $S$ be a noetherian scheme. The \emph{Morel-Voevodsky stable homotopy category} $SH(S)$ is the universal tensor triangulated category such that:
\begin{enumerate}
 \item There is a monoidal functor $\Sigma^\infty(-)_+: \Sm(S) \to SH(S)$.
 \item ($\AA^1$-invariance) $\Sigma^\infty(\AA^1_S)_+ \to \Sigma^\infty(S)_+$ is an isomorphism.

 \item (Nisnevich descent) $\Sigma^\infty(-)_+$ sends Nisnevich distinguished squares$^\textrm{\ref{footnote:nisDistSq}}$ to homotopy cocartesian squares. That is, in the notation of Footnote~\ref{footnote:nisDistSq}, the Mayer-Vietoris style triangle
\begin{equation}
  \Sigma^\infty(U{\times_X}V)_+ \stackrel{}{\longrightarrow} \Sigma^\infty U_+ \oplus \Sigma^\infty V_+ \stackrel{}{\longrightarrow} \Sigma^\infty X_+ 
\end{equation}
fits into a distinguished triangle.

 \item ($(\PP^1, \infty)$-Stability) The cofibre of the image of the section at infinity $\infty: S  \to \PP^1_S$ is tensor invertible. That is, $\Cone\bigl (\Sigma^\infty S_+ \stackrel{}{\to} \Sigma^\infty (\PP^1_S)_+ \bigr)$ is tensor invertible.
\end{enumerate}
\end{heurDefi}

The tensor triangulated category $SH(S)$ can be constructed as follows.

\begin{cons}
\cite[1.2, 2.15]{Cis13} Let $Sp_{S^1}(S)$ denote the category of presheaves of symmetric $S^1$-spectra on $\Sm(S)$. This category is equipped with the projective model structure.%
\footnote{The projective model structure is the model structure for which a morphism is a fibration (resp. weak equivalence) if and only if it is a fibration (resp. weak equivalence) of symmetric $S^1$-spectra after evaluation on every $X \in \Sm(S)$.} %
Via Yoneda, every smooth $S$-scheme equipped with a section, such as $\PP^1_S$ equipped with the section at infinity, gives rise to an object of $Sp_{S^1}(S)$.%
\footnote{\label{footnote:spYon} The Yoneda functor produces a presheaf of pointed sets $\hom_{\Sm(S)}(-, \PP^1_S)$, and then working schemewise, we associated to every pointed set its induced pointed simplicial set, and from there its associated symmetric $S^1$-spectrum.} %
Let $T$ be an cofibrant replacement%
\footnote{The projective model structure has the nice property that representable presheaves are cofibrant, however, by representable we mean the image of a scheme with a disjoint base point, cf. Footnote~\ref{extrabasepoint}. Presheaves which are the image of pointed schemes whose basepoint is not disjoint are not in general cofibrant. Hence, we need to take some cofibrant model. For example, the pushout of $S_+ \wedge \Delta^1_+ \stackrel{0}{\leftarrow} S_+ \stackrel{\infty}{\to} \PP^1_+$ is a cofibrant model for $(\PP^1_S, \infty)$, where $\Delta^1_+$ is the constant presheaf corresponding to the simplicial interval. This is exactly the analogue of the complex $([S] \stackrel{\infty}{\to} [\PP^1_S])$ which we could have used to define $\ZZ(1)$ in $DM_{gm}^{eff}(S)$.} %
 for the pointed scheme $\PP^1_S$ in $Sp_{S^1}(S)$, and let $Sp_TSp_{S^1}(S)$ denote the category of symmetric $T$-spectra in $Sp_{S^1}(S)$. It comes equipped with a canonical ``Yoneda'' functor%
\footnote{\label{extrabasepoint}We first equip any $S$-scheme $X$ with a base point by replacing it with $X_+ = X \amalg S$. Then using the procedure described in Footnote~\ref{footnote:spYon} we get a functor $\Sm(S) \to Sp_{S^1}(S)$, which we compose with the canonical functor $Sp_{S^1}(S) \to Sp_TSp_{S^1}(S)$.}
\[ \Sigma^\infty(-)_+: \Sm(S) \to Sp_TSp_{S^1}(S), \]
and a ``constant sheaf'' functor%
\footnote{This is actually the composition of the constant presheaf functor $Sp_{S^1} \to Sp_{S^1}(S)$ from symmetric $S^1$-spectra to presheaves of symmetric $S^1$-spectra, and the canonical functor $Sp_{S^1}(S) \to Sp_TSp_{S^1}(S)$.}%
\begin{equation} \label{equa:constFunctor}
 Sp_{S^1} \to Sp_TSp_{S^1}(S).
 \end{equation}
Let $\HI$ be the set of images in the homotopy category $\Ho(Sp_TSp_{S^1}(S))$ of the morphisms $\Sigma^\infty (\AA^1_X)_+ \to \Sigma^\infty X_+$ as $X$ ranges over all smooth $S$-schemes. Let $\NMV$ be the set of images in $\Ho(Sp_TSp_{S^1}(S))$ of the morphisms $\Cone(\Sigma^\infty (U {\times_X} V)_+ \to \Sigma^\infty U_+ \oplus \Sigma^\infty V_+) \to \Sigma^\infty X_+$ ranging over all distinguished Nisnevich squares$^\textrm{\ref{footnote:nisDistSq}}$ in $\Sm(S)$. We define $SH(S)$ as the Verdier quotient
\[ SH(S) = \frac{\Ho(Sp_TSp_{S^1}(S))}{\langle\!\langle \HI \cup \NMV \rangle\!\rangle}. \]
Here the double angle brackets $\langle\!\langle - \rangle\!\rangle$ indicate the localising category, i.e., the smallest triangulated category closed under direct summands and arbitrary small sums, containing the objects of $\HI$ and $\NMV$.
\end{cons}

Its subcategory of compact objects is the smallest thick triangulated category containing the objects $\Sigma^\infty X_+$ for all $X \in \Sm(S)$. This is denoted by
\[ \SH(S)^c = \langle \Sigma^\infty X_+ : X \in \Sm(S) \rangle \subset SH(S) \]

\begin{rema}
There are also constructions using $S^1 \wedge \GG_m$ (where $\GG_m$ is pointed at the identity) instead of $\PP^1_S$, and constructions which take the $\HI \cup \NMV$ localisations as Bousfield localisations before passing to $T$-spectra. These all produce equivalent categories, \cite[2.15]{Cis13}. 
\end{rema}

\begin{rema} \label{DoldKan}
The canonical inclusion$^\textrm{\ref{footnote:smcordisclaimer}}$ $\Sm(S) \to \SmCor(S)$, and the canonical functor $Sp_{S^1} \to \Cpx(\Ab)$ from symmetric $S^1$-spectra to (unbounded) complexes of abelian groups induces a canonical tensor triangulated functor
\[ SH(S) \to DM(S). \]
Here, $DM(S)$ can defined in the same way as we have defined $SH(S)$, but using $\SmCor(S)$ instead of $\Sm(S)$, and $\Cpx(\Ab)$ instead of $Sp_{S^1}$. As $DM_{gm}(S)$ can be identified with the thick subcategory of $DM(S)$ generated by the motives of smooth schemes, this functor restricts to a functor
\[ SH(S)^c \to DM_{gm}(S). \]
\end{rema}

\begin{theo}[{Morel, \cite{Mor06}, cf. also \cite[§A.3, §C.3]{CD12}}] \label{theo:morel}
If $S$ is the spectrum of a field such that $-1$ is a sum of squares (such as a finite field, or the field of complex numbers), the canonical functor, cf. Remark~\ref{DoldKan},
\[ SH(S)^c_\QQ \to DM_{gm}(S)_\QQ \]
is a $\otimes$-triangulated equivalence of categories.
\end{theo}

\begin{rema}
Morel proved this theorem is true for the rational versions of the big categories $SH(S)$ and $DM(S)$, but Definition~\ref{defi:tensorLambda} does not work properly for noncompact objects, so some care must be taken with the term ``rational version''. Either one can localise at all morphisms $n \cdot \id_M$ for all objects $M$ and all integers $n \neq 0$, or since the rational versions of $Sp_S^1$ and $\Cpx(\Ab)$ are Quillen equivalent to the category of unbounded complexes of $\QQ$-vector spaces, one could just use these in the constructions instead.
\end{rema}

The study of the motivic stable homotopy groups of spheres, i.e., $\hom_{SH(S)}(\un[p + q], T^q)$ for $p, q \in \ZZ$ is one of the central problems in motivic homotopy theory.

\begin{theo} \label{theo:motStabHomGrps}
Suppose that $k$ is the spectrum of a perfect field.
\begin{enumerate}
 \item $\hom_{SH(k)}(\un[p {+} q], T^q) = 0$ if $p < 0$, \cite[Thm{.}4.2.10, \S6 Intro{.}]{Mor04}.

 \item The graded ring $\oplus_{q \in \ZZ} \hom_{SH(k)}(\un[q], T^q)$ is canonically isomorphic to the graded associative ring $K_\bullet^{MW}(k)$ generated by symbols $[a], a \in k^*$, of degree $1$, and one symbol $\eta$ of degree $-1$, subject to the following relations.
 \begin{enumerate}
 \item $[a][1-a] = 0$.
 \item $[ab] = [a] + [b] + \eta[a][b]$.
 \item $\eta[a] = [a]\eta$.
 \item $\eta \cdot h = 0$, where $h = 1 + (\eta[-1] + 1)$, \cite[Def{.}12, Cor{.}25]{Mor12}.
\end{enumerate}
In degree zero, we have the Grothendieck-Witt group $GW(k) \cong K^{MW}_0(k)$, \cite[Cor{.}24]{Mor12}. If the characteristic of $k$ is not 2, then for every $n \in \ZZ$, there is a canonical short exact sequence of abelian groups
\begin{equation} \label{equa:MWses}
 0 \to I(k)^{n + 1} \to K^{MW}_n(k) \to \underbrace{K^M_n(k)}_{\cong K^{MW}_n(k) / \eta} \to 0 
 \end{equation}
where $I(k) = \ker(W(k) {\to} \ZZ / 2)$ is the augmentation ideal in the Witt ring $W(k)$ of the field $k$, we set $I(k)^n = W(k)$ for $n < 0$ by convention, and $K^M_n(k)$ is the Milnor $K$-theory of the field $k$, \cite{Mor04SurLes}, \cite[Def{.}3.7, Thm{.}3.8, Thm{.}5.4]{GSZ}\footnote{There are some weird sign and bracket conventions in \cite{GSZ}. In \cite{GSZ}, $\hat{\eta}$ and $\{a\}$ are used to denote what \cite{Mor12} writes as $\eta$ and $[a]$. On the other hand, \cite{GSZ} use $\eta$ and $[a]$ for the elements in $K^{W} \stackrel{def}{=} K^{MW} / h$ corresponding to our $\eta$ and $-[a]$.}.

 \item If $k$ is algebraically closed of characteristic zero then $\hom_{SH(k)}(\un[n], \un) \cong \pi^s_n$, the classical stable homotopy groups, \cite[Cor{.}2]{Lev14}. If $k$ is an algebraic closure of a finite field with $p$ elements, and $l$ is a prime different from $p$, then the we have the same result after $l$-completion $\hom_{SH(k)}(\un[n], \un)_l^{\wedge} \cong (\pi^s_n)_l^{\wedge}$, \cite[Thm{.}A]{WO16}. In fact, for any subfield $k \subseteq \CC$, the $\otimes$-triangulated functor $SH_\top \to SH(k)$ induced by the constant presheaf funtor, cf. Equation~\eqref{equa:constFunctor}, is fully faithful.
 
 \item \label{theo:motStabHomGrps:split} In fact, let $k \subseteq \CC$ be a subfield. Then sending a smooth $k$-scheme $X$ to its complex valued points $X(\CC)$ considered as a topological space in the obvious way induces a $\otimes$-triangulated functor $SH(k) \to SH_\top$ towards the classical stable homotopy category, \cite{Ayo10}, which is a retraction of the $\otimes$-triangulated functor $SH_\top \to SH(k)$ induced by the constant presheaf funtor, cf. Equation~\eqref{equa:constFunctor}.

 \item There are canonically defined ``Hopf'' elements $\eta \in \hom(\un[-1], T^{-1})$, $\nu \in \hom(\un[-1], T^{-2})$, $\sigma \in \hom(\un[-1], T^{-4})$, satisfying the relations $(1 - \epsilon)\eta = \eta \nu = \nu\sigma = 0$, where $\epsilon: \un[-1] \to T^{-1}$ corresponds to the map $\GG_m \otimes \GG_m \to \GG_m \otimes \GG_m$ swapping the factors, \cite[Def{.}4.7]{DI13}.
\end{enumerate}
\end{theo}

\begin{rema}
For more calculations about motivic stable homotopy groups of spheres, see, for example, \cite{DI10}, \cite{HKO}, \cite{OO13}, \cite{OO14}, and the references therein.
\end{rema}

\begin{rema} \label{rema:kmwdescr}
In the $K_\bullet^{MW}(k)$ description of $\oplus_{q \in \ZZ} \hom_{SH(k)}(\un[q], T^q)$, the symbol $[a]$ corresponds to the section $k \to \GG_m$ associated to the rational point $a \in \GG_m(k)$, and the symbol $\eta$ corresponds to the Hopf map $\AA^2 {-} \{0\} \to \PP^1$, under the isomorphisms \cite[Lem{.} 4.1]{Voe98}
\[ (\AA^n {-} \{0\}, 1) \cong T^n[-1], \qquad \PP^n / \PP^{n - 1} \cong \AA^n / (\AA^n {-} \{0\}) \cong T^n. \] 
The element $h$ corresponds to the class of the hyperbolic plane in $GW(k) \cong K_0^{MW}(k)$, \cite[\S 6.2, \S 6.3]{Mor04}.
\end{rema}

\begin{rema} \label{rema:MWdescription}
The maps $I^{n + 1} \to K_n^{MW}$ in the short exact sequence \eqref{equa:MWses} are as follows. First, there is an isomorphism of graded rings $\oplus_{n \in \ZZ} I^n \cong K^W_\bullet \stackrel{def}{=} (\oplus _n K^{MW}_n) / h$ where $[a] \in K^{W}_1$ corresponds to the Pfister form $\langle 1, -a \rangle \in I$ and $1 + \eta [a] \in K^{MW}_0$ corresponds to the one dimensional quadratic space $\langle a \rangle \in W$. For $n \leq 0$, multiplication by $\eta$ induces the isomorphism $K^W_n \to K^W_{n - 1}$ corresponding to $W^n = W^{n - 1}$ \cite[p{.}13]{GSZ}. Next, since $\eta \cdot h = 0$, multiplication by $\eta$ induces a map $K^W_\bullet = K^{MW}_\bullet / h \to  K^{MW}_{\bullet - 1}$. The map in question is the composition
\[ I^{n + 1} \cong K^W_{n + 1} \stackrel{-\cdot \eta}{\to} K^{MW}_{n}. \]
\end{rema}

\begin{rema}
The prime homogeneous ideals of $K^{MW}(k)$ for any field $k$ of characteristic not 2 are classified in \cite{Tho15}.
\end{rema}

\begin{exam} \label{exam:finiteFieldMW}
Let $k$ be a finite field with $q$ elements, and $q$ odd. Then 
%
\[ K^{MW}_{\geq 0}(\FF_q) \cong \biggl ( \ZZ \oplus \ZZ / 2 \biggr ) \oplus \FF_q^\ast \oplus 0 \oplus 0 \oplus \dots, \] 
and for $n > 0$ we have
\[ K^{MW}_{-n}(\FF_q) = \left \{ \begin{array}{ll} \ZZ / 4 & q \equiv 3 \mod 4, \\ \ZZ / 2[\epsilon] / \epsilon^2 & q \equiv 1 \mod 4, \end{array} \right . \]
\cite[p{.}37]{Lam05}, \cite[p{.}36, Thm{.}3.5]{Lam05}, \cite[Exam{.}1.5]{Mil69}.

Let $\omega$ be a multiplicative generator for $\FF_q^*$. Studying Remark~\ref{rema:MWdescription} we find that $K^{MW}_1$ is additively generated by $[\omega]$, the copy of $\ZZ / 2$ in $K^{MW}_0$ is additively generated by $\eta \cdot [\omega]$, and for $n > 0$, $K^{MW}_{-n}$ is additively generated by $\eta^n$ if $q \equiv 3 \mod 4$, or by $\eta^n$ and $\eta^{n + 1}[\omega]$ if $q \equiv 1 \mod 4$. 

If $\FF$ is the algebraic closure of $\FF_q$, then we have 
\[ K_n^{MW}(\FF) \cong \dots \ \oplus\  \underset{-3}{\ZZ / 2} \ \oplus\  \underset{-2}{\ZZ / 2} \ \oplus\  \underset{-1}{\ZZ / 2} \ \oplus\  \underset{0}{\ZZ} \ \oplus\  \overset{\cong \hat{\ZZ}}{\underset{1}{\FF^*}} \ \oplus\  \underset{2}{0} \ \oplus\  \underset{3}{0} \ \oplus\  \underset{4}{0} \ \oplus\  \dots. \]

\end{exam}

\section{Observations}
\label{sec:observations}

In this section we make some observations and guesses about the structure of $\Spc(SH(\FF_q)^c)$.

\subsection{Rational coefficients}

\begin{theo} \label{theo:main}
Let $\FF_q$ be a field with a prime power, $q$, number of elements. Suppose that for all connected smooth projective varieties $X$ we have:
\[
\begin{array}{rllr}
CH^i(X, j)_\QQ &= 0,  & \forall j \neq 0, i \in \ZZ & \textrm{(Beilinson-Parshin conjecture)}, \\
CH^i(X)_\QQ \otimes CH_i(X)_\QQ &\to CH_0(X)_\QQ & \textrm{ is non-degenerate.} & \textrm{(Rat.\ and num.\ equiv.\ agree)} 
\end{array}
\]
Then 
\[ \Spc(SH(\FF_q)^c_\QQ)  \cong \Spec(\QQ).\]
\end{theo}

\begin{proof}
First we observe that by Theorem~\ref{theo:morel} of Morel, the canonical functor $SH(\FF_q)^c_\QQ \to DM_{gm}(\FF_q)_\QQ$ is an equivalence. So we are reduced to studying $DM_{gm}(\FF_q)_\QQ$. 

On the other hand, since we are assuming that rational and numerical equivalence agree, the canonical functor $K^b(\mathcal{M}_{\rat}(\FF_q)_\QQ) {\to} K^b(\mathcal{M}_{\num}(\FF_q)_\QQ)$ is an equality. Let us write $\cM(\FF_q)_\QQ$ for the category $\mathcal{M}_{\rat}(\FF_q)_\QQ {=} \mathcal{M}_{\num}(\FF_q)_\QQ$. Jannsen's semisimplicity theorem, Theorem~\ref{theo:jannsen}, says that $\mathcal{M}(\FF_q)_\QQ$ is semisimple, and therefore $K^b(\mathcal{M}(\FF_q)_\QQ) \cong \bigoplus_{i \in \ZZ} \mathcal{M}(\FF_q)_\QQ$ as a tensor triangulated category, cf. the proof of Lemma~\ref{lemm:rigssfield}. It follows that the canonical (tensor) functor $\mathcal{M}_{\rat}(\FF_q)_\QQ \to DM_{gm}(\FF_q)_\QQ$ of Theorem~\ref{theo:dmtheo}\eqref{theo:dmtheo:3} extends to a (tensor) triangulated functor 
\begin{equation} \label{eq:canFunMDM}
\bigoplus_{i \in \ZZ} \mathcal{M}(\FF_q)_\QQ \to DM_{gm}(\FF_q)_\QQ. 
\end{equation}
For smooth projective varieties $X$ and $Y$, the Beilinson-Parshin conjecture says that 
\[ \hom_{DM_{gm}(\FF_q)_\QQ}(M(X), M(Y)[i]) \stackrel{Thm{.}\ref{theo:dmtheo}\eqref{theo:dmtheo:2}}{\cong} CH^{\dim Y}(X {\times} Y, -i) \]
vanishes unless $i = 0$. It follows that the functor \eqref{eq:canFunMDM} is fully faithful. Since $DM_{gm}(\FF_q)_\QQ$ is generated by motives of smooth projective varieties, Theorem~\ref{theo:dmtheo}\eqref{theo:dmtheo:4}, the functor \eqref{eq:canFunMDM} is also essentially surjective. That is, it is an equivalence of tensor triangulated categories. Finally we apply Lemma~\ref{lemm:rigssfield} and Remark~\ref{rema:rigidity} to notice that $\cM(\FF_q)_\QQ$ has a unique prime: the zero prime. Consequently, the same is true for $DM_{gm}(\FF_q)_\QQ$. For the structure sheaf: we have $\hom_{\cM_{\rat}(\FF_q)_\QQ}(\un, \un) \cong \QQ$ by definition.
\end{proof}

\begin{rema}
If one is willing to accept that Bondarko's weight complex functor $DM_{gm}(\FF_q)_\QQ \stackrel{t}{\to} K^b(\mathcal{M}_{\rat}(\FF_q)_\QQ)$ is monoïdal, there is a much more conceptual approach to the above proof. One considers the sequence of monoïdal functors
\[ SH(\FF_q)^c_\QQ \to DM_{gm}(\FF_q)_\QQ \stackrel{t}{\to} K^b(\mathcal{M}_{\rat}(\FF_q)_\QQ) {\to} K^b(\mathcal{M}_{\num}(\FF_q)_\QQ). \]
The first one is an equivalence by Morel, the second one is an equivalence if and only if the Beilinson-Parshin conjecture holds, \cite{Bon09}, and the third one is an equivalence if and only if $\rat = \num$. We didn't use this because we didn't want to claim without proof that $t$ is monoïdal, we didn't want to include a proof, and we don't know a reference for this statement.
\end{rema}

\begin{obse}
Conversely, if $\Spc(SH(\FF_q)^c_\QQ) \cong \Spec(\QQ)$, then the étale realisation functor of Theorem~\ref{theo:etalereal}
\[ R: DM_{gm}(\FF_q)_\QQ^{op} \to D^b_{et}(\FF_q, \QQ_l) \]
is conservative for any $l \nmid q$. That is, for any object $M$, we have $R(M) \cong 0$ if and only if $M \cong 0$.
\end{obse}

The same is true if the base is the complex numbers, and we use the Betti realisation instead. Consequences of this conservativity are explored in \cite[\S2]{Ayo15}.\footnote{In fact, the Conservativity Conjecture described in \cite{Ayo15} is one of the major open problems in the area. It is a kind of triangulated analogue of the Hodge conjecture.}

%

\begin{conj}[Cisinski]
If $\Spc(SH(\FF_q)^c_\QQ) \cong \QQ$, then the Beilinson-Parshin conjecture is true, and rational and numerical equivalence agree.
\end{conj}

\subsection{The structural morphism}

\begin{prop} \label{prop:epi}
Let $q$ be an odd prime power 
and consider the open-closed decomposition of the topological space $\Spc(SH(\FF_q)^c) = U(\PP^2) \cup \supp(\PP^2)$, cf. Definition~\ref{defi:basicDef}\eqref{defi:basicDef:supp}, where $\PP^2$ is pointed at any rational point. The canonical morphism of \emph{topological spaces}
\[ \supp(\PP^2) \to \Spec(\ZZ), \qquad \textrm{resp. } U(\PP^2) \to \Spec(\ZZ) \]
is surjective, resp. has image $(2) \in \Spec(\ZZ)$. Moreover, the closure of $U(\PP^2)$ in $\Spc(SH(\FF_q)^c)$ intersects $\supp(\PP^2)$ nontrivially.
\end{prop}

\begin{proof}
We begin with $U(\PP^2)$. First we claim that $U(\PP^2)$ is non-empty, cf. \cite[Proof of Prop.10.4]{BalSSS10}. Recall that $\eta: \un[-1] \to T^{-1}$ is the morphism induced by the canonical morphism $\AA^2{-}\{0\} \to \PP^1$, and note that $\PP^2 \cong Cone(\eta)$ by the Mayer-Vietoris distinguished triangle associated to $\{\PP^2 {-}\{ x\} \to \PP^2, \PP^2{-}\PP^1 \to \PP^2\}$ (and $\AA^1$-invariance) where $x$ is any rational point not in $\PP^1$. So by Proposition~\ref{prop:firstProp}\eqref{prop:firstProp:nil} the set $U(\PP^2)$ is non-empty if and only if $\eta^{\otimes n}$ is nonzero for all $n > 0$. But this latter follows from the fact that $\eta$ is not nilpotent in $K^{MW}(\FF_q)$, Theorem~\ref{theo:motStabHomGrps}, Example~\ref{exam:finiteFieldMW}, (this nonnilpotence is true for any field \cite[Cor.6.4.5,p.258]{Mor04MotPi}).

By Proposition~\ref{prop:firstProp}\eqref{prop:firstProp:open} the subspace $U(\PP^2)$ is isomorphic to $\Spc\left ( SH(\FF_q)^c / \langle \PP^2 \rangle \right )$. Since $\PP^2 \cong Cone(\eta)$, the morphism $\eta$ is invertible in $SH(\FF_q)^c / \langle \PP^2 \rangle$, and therefore the canonical morphism of graded rings $\oplus_{n \in \ZZ} \hom_{SH(\FF_q)}(\un[n], T^n) \to \oplus_{n \in \ZZ} \hom_{SH(\FF_q) / \langle \PP^2 \rangle}(\un[n], T^n)$ factors through the localised ring $\oplus_{n \in \ZZ} \hom_{SH(\FF_q)}(\un[n], T^n)[\eta^{-1}]$. It follows from the description of Example~\ref{exam:finiteFieldMW} that $4$ is zero in this latter. Hence $4 \cdot \id_\un$ is zero in $SH(\FF_q) / \langle \PP^2 \rangle$ and therefore the canonical morphism of topological spaces $U(\PP^2) \cong \Spc(\left ( SH(\FF_q) / \langle \PP^2 \rangle \right )^c) \to \Spec(\ZZ)$ given by Theorem~\ref{theo:structureMap} factors through $\Spec(\ZZ / 2) \subset \Spec(\ZZ)$.

Now consider $\supp(\PP^2)$. Using Theorem~\ref{theo:structureMap} and Theorem~\ref{theo:motStabHomGrps} again, we find a surjective morphism of topological spaces 
\[ \Spc(SH(\FF_q)^c) \to \Spec(GW(\FF_q)) \stackrel{\textrm{Exm.\ref{exam:finiteFieldMW}}}{\cong} \Spec(\ZZ). \]
We have seen that the open subspace $U(\PP^2)$ is sent to the closed subspace $(2)$. So to see that $\supp(\PP^2) \to \Spec(\ZZ)$ is still surjective, it suffices to show that $\supp(\PP^2)$ intersects the closure of $U(\PP^2)$ nontrivially. That is, it suffices to show that the topological space $\Spc(SH(\FF_q)^c)$ is not the disjoint union of the topological spaces $U(\PP^2)$ and $\supp(\PP^2)$. Such a decomposition would induce a decomposition of the ring $End_{SH(\FF_q)^c_{\ZZ[1/p]}}(\un)$ into the direct product of the non-zero rings $End_{U(\PP^2)}(\un)[1/p]$ and $End_{\supp(\PP^2)}(\un)[1/p]$, \cite[Theorem 2.11]{Bal07}, \cite[Riou's Appendix B]{LYZ13}. 
Since $1 \in GW(\FF_q)[1/p]$ is not a sum of two non-zero idempotents, this is not possible. 
\end{proof}

\begin{coro} \label{coro:compMapEpi}
Let $q$ be an odd prime power. The canonical morphism
\[ \Spc(SH(\FF_q)^c) \to \Spec^h(K^{MW}(\FF_q)) \]
of Theorem~\ref{theo:structureMap} and Theorem~\ref{theo:motStabHomGrps} is surjective.
\end{coro}

\begin{rema}
This surjectivity is proven for a general field by Heller and Ormsby in \cite{HO16} using the classification of the prime ideals of $K^{MW}$ from \cite{Tho15}. It was Ormsby's idea of using this classification which lead to our proof.
\end{rema}

\begin{proof}
One can classify the homogenous prime ideals of $K^{MW}(\FF_q)$ by hand%
\footnote{They are in bijection with the homogeneous prime ideals of $K^{MW}(\FF_q)_{red} \cong \ZZ[\eta] / 2 \cdot \eta$. Then we have $\Spec(\ZZ[\eta] / 2 \cdot \eta) = \Spec(\ZZ) \cup \Spec(\ZZ / 2[\eta])$. Clearly, apart from $\Spec(\ZZ / 2) = \Spec(\ZZ) \cap \Spec(\ZZ / 2[\eta])$, the graded ring $\ZZ / 2[\eta]$ has exactly one other homogeneous prime: $(\eta)$.} %
using the description in Example~\ref{exam:finiteFieldMW}, or consult \cite{Tho15}, and find that they are: $([\omega], \eta, p)$ with $p \in \ZZ$ an odd prime, $([\omega], 2)$, and $([\omega], \eta)$.

It follows from $U(\PP^2) = U(Cone(\eta))$ being non-empty that $([\omega], 2)$ is in the image of $\Spc(SH(\FF_q)^c)$. The primes $([\omega], \eta, p)$ with $p \in \ZZ$ an odd prime are in the image by the surjectivity of $\Spc(SH(\FF_q)^c) \to \Spec(\ZZ)$. Finally, $([\omega], \eta)$ is in the image by the claim that $\supp(\PP^2)$ intersects the closure of $U(\PP^2)$ nontrivially.
\end{proof}

\subsection{Equivariant stable homotopy theory} \label{subsec:equi}

Much information about $SH(\RR)$ comes from the realisation functor, $SH(\RR) \to SH_{\ZZ / 2}$, towards the $\ZZ / 2$-equivariant stable homotopy category, cf. Theorem~\ref{theo:motStabHomGrps}\eqref{theo:motStabHomGrps:split}. The étale homotopy type, \cite{AM69}, \cite{Fri73}, provides a functor from $\Sm(S)$ to the homotopy category $\Ho(\ProSS)$ of pro-finite simplicials sets. We hope that this induces a tensor triangulated functor towards an appropriate $Gal(\FF_q)$-equivariant stable homotopy category. 
The structure (as a set) of the spectrum of the equivariant stable homotopy category, $\Spc(SH_G^c)$, of a finite group $G$ has recently been completely determined by Balmer-Sanders, \cite{BS15}. They also describe much about its topology. This suggests that the $Gal(\FF_q)$-equivariant stable homotopy category mentioned above has a good chance of being accessible, and providing information about the structure of $\Spc(SH(\FF_q)^c_{\ZZ_{(l)}})$.

\subsection{Final observations}

Recall the following.

\begin{enumerate}
 \item It seems highly likely that 
 \[ \Spc(SH(\FF_q)^c_\QQ) \cong \Spec(\QQ). \]

 \item Let $\FF$ be the algebraic closure of the finite field with an odd number of elements. The endomorphism ring of the unit $GW(\FF) \cong End_{SH(\FF)}(\un)$ and the graded endomorphism ring $K^{MW}(\FF) \cong \oplus_{n \in \ZZ} \hom_{SH(\FF)}(\un[n], T^{n})$ are extremely simple, especially if we invert 2, and ignore nilpotents. Indeed, 
 \[ K^{MW}(\FF)[1/2]_{red} \cong \ZZ[1/2]. \]
 
 \item Due to the realisation functor, for a subfield $k \subseteq \CC$, there is a retraction
 \[ \xymatrix{\Spc(\SH_\top^c) \ar[r] \ar@/^12pt/[rr]^{\cong} & \Spc(SH(k)^c) \ar[r] & \Spc(\SH_\top^c)}.
 \]
 It seems likely that the étale realisation should give rise to a similar phenomenon for all fields.
\end{enumerate}
%
%
%

Based on these observations, lets 
make the following wild speculation. We have taken the algebraic closure of the base field to avoid equivariant phenomena which might appear, as discussed in Section~\ref{subsec:equi}, but one could also ask if the case of a finite field base is just the appropriate combination of the $\FF$-base case together with $\Gal(\FF / \FF_p)$-equivariant phenomena, cf. Example~\ref{exam:spc}\eqref{exam:spc:equi}.

\begin{gues} \label{gues}
Let $\FF$ be an algebraic closure of a finite field with $p$ elements. Then the canonical constant presheaf functor $SH_{\top} \to SH(\FF)$ from the classical stable homotopy category, cf. Equation~\eqref{equa:constFunctor}, induces an isomorphism
\[ \Spc(SH(\FF)^c_{\ZZ_{(l)}}) \to \Spc((SH_{\top}^c)_{\ZZ_{(l)}}) \]
whenever $l$ and $p$ are odd.
\end{gues}


\bibliographystyle{alpha}
\bibliography{bib}

\end{document}